\documentclass[a4paper,11pt]{amsart}
\usepackage[a4paper,left=3cm,right=3cm,top=3cm,bottom=3cm]{geometry}
\usepackage{amsmath, amssymb, amsthm}
\numberwithin{equation}{section} 

\usepackage{tikz-cd}
\usepackage{newtxtext,newtxmath}
\usepackage{mathtools}
\usepackage{physics}
\usepackage{esint}
\usepackage{subcaption}
\usepackage{graphicx}
\usepackage{enumitem}
\usepackage{xcolor}
\usepackage{hyperref}
\usepackage{cleveref}
\usepackage{comment}
\usepackage{orcidlink}
\usepackage{needspace}
   
\title[Fundamental Principle for $\lambda$-Homogeneous Solutions]{Towards a Fundamental Principle for \\$\lambda$-Homogeneous Solutions on Cones}
\author[M. Tsopanopoulos]{
Michael Tsopanopoulos\, \orcidlink{0009-0007-3167-6862} \\
\texttt{tsopanopoulos@wias-berlin.de} }
\thanks{Weierstrass Institute for Applied Analysis and Stochastics, Anton-Wilhelm-Amo-Straße 39, 10117 Berlin, Germany.}

\newtheorem{theorem}{Theorem}[section]
\newtheorem{lemma}[theorem]{Lemma}
\newtheorem{proposition}[theorem]{Proposition}

\newtheorem{assumption}{Assumption}

\theoremstyle{definition}
\newtheorem{definition}[theorem]{Definition}
\newtheorem{example}[theorem]{Example}
\theoremstyle{remark}
\newtheorem*{remark}{\textbf{Remark}}

\begin{document}

\begin{abstract}
We prove a weak fundamental principle for \(\lambda\)-homogeneous solutions of homogeneous constant-coefficient systems on open pointed convex cones. Starting with the solution family \(S_{\mathcal B}\) arising in the Ehrenpreis--Palamodov theory, we construct a corresponding family \(S_{\mathcal B,\lambda}\) by replacing the exponential kernels \(e^{\langle x,z\rangle}\) with homogeneous kernels \((-\langle x,z\rangle)^\lambda\). The key tool is a Mellin-type operator on Paley--Wiener spaces, which links the classical theory to the Euler-constrained setting. For $\lambda\in \mathbb{C}\setminus \mathbb{N}_0$ and under a visibility assumption, we show that the span of \(S_{\mathcal B,\lambda}\) is dense in the space of \(\lambda\)-homogeneous solutions.
\end{abstract}

\maketitle


\medskip


\section{Introduction}
Let $\Omega\subset \mathbb{R}^n$ be open and convex, and let \(P\in\mathbb C[z]^{J\times L}\) be a $J\times L$ matrix of polynomials in $z\in \mathbb{C}^n$. The fundamental principle for linear PDEs with constant coefficients describes solutions of 
\begin{align*}
    P(\partial_x)u=0\quad \text{on }\Omega
\end{align*}
in terms of exponential solutions $S_e(\Omega)$. These are solutions which admit a representation of the form 
\begin{align*}
    u(x)= p(x)e^{\langle x,z\rangle},\quad z\in\mathbb C^n,
\end{align*}
where $p\in \mathbb{C}[x]^L$ is a vector of polynomials. A weak form of this principle, due to Malgrange \cite{Malgrange1955}, states that the span of $S_e(\Omega)$ is dense in the solution space \(\ker(P(\partial_x))\). A strong form is provided by the celebrated Ehrenpreis--Palamodov representation theorem, which shows that every solution admits an integral representation in terms of a suitable solution family $S_{\mathcal{B}}(\Omega)\subset S_e(\Omega)$, determined by finitely many Noetherian operators. See the classical monographs of Ehrenpreis \cite{Ehrenpreis1970} and Palamodov \cite{Palamodov1970}, as well as Hörmander’s treatment in \cite{hörmandercomplex}, and the survey \cite{Hansen1983}.

We now impose an additional Euler constraint. Assume that $P$ is homogeneous of degree $m\in \mathbb{N}$ and that $\Omega\subset \mathbb{R}^n$ is an open convex cone. For $\lambda\in \mathbb{C}$, let $E^x_\lambda = x\cdot \nabla_x -\lambda$ denote the Euler operator. Our object of interest is the space
\begin{align}\label{solspaces}
    \operatorname{ker}(P(\partial_x))\cap \operatorname{ker}(E^x_\lambda).
\end{align}
On a cone, \(\ker(E_\lambda^x)\) is precisely the space of \(\lambda\)-homogeneous functions, that is, functions satisfying $u(tx) = t^\lambda u(x)$ for $t>0$. Thus, (\ref{solspaces}) consists of $\lambda$-homogeneous solutions of $P(\partial_x)u=0$. Although $\operatorname{span} S_{\mathcal{B}}(\Omega)$ is dense in \(\ker(P(\partial_x))\), its elements generally do not satisfy the Euler constraint. The aim of this paper is to construct a $\lambda$-homogeneous analogue and to prove a weak fundamental principle for the $\lambda$-homogeneous solution space. Here, the central idea is to replace the exponential kernel $e^{\langle x,z\rangle}$ by the homogeneous kernel
\begin{align}\label{newexp}
    \Omega \ni x \mapsto (-\langle x,z\rangle)^\lambda.
\end{align}
The Mellin transform provides the link between the two kernels. Indeed, for $\Re\lambda<0$ and $\Re \beta>0$, one has
\begin{align}\label{intro-mellin}
 \beta^\lambda =\frac{1}{\Gamma(-\lambda)} \int_0^\infty e^{-t\beta} t^{-\lambda-1}\,dt,
\end{align}
where $\Gamma$ denotes the gamma function; see Equation~(5) in \cite{ZagierAppendix}. Substituting \(\beta=-\langle x,z\rangle\) relates \(e^{\langle x,z\rangle}\) to \((-\langle x,z\rangle)^\lambda\). In the sequel, this relation is implemented by a Mellin-type operator $M_\lambda^z$ acting on holomorphic functions of Paley--Wiener growth. A regularized version of (\ref{intro-mellin}) extends the construction to all $\lambda\in \mathbb{C}\setminus \mathbb{N}_0$.

Beyond extending the classical theory to include the Euler constraint, $\lambda$-homogeneous solutions are of independent interest in regularity theory. In particular, information on the possible values of $\Re \lambda$ plays an important role in the analysis of singular behavior; see, for example, \cite{Daug1988EBVP, GrisEPND2011, KMR1997EBVP, MaRO2010EEPD, VMR2001SPCS}. In \cite{Tsopanopoulos2025}, such questions were studied for elliptic systems on an angle, that is, a two-dimensional conic domain, where kernels of the form (\ref{newexp}) appeared naturally. From this perspective, the present work can be viewed as a continuation in arbitrary dimensions $n\geq 2$, and for general homogeneous systems.

The paper is organized as follows. Section~\ref{review} recalls the classical fundamental principle, the associated affine characteristic varieties $V_{\mathfrak p_\sigma}$, and the Noetherian exponential family $S_{\mathcal{B}}(\Omega)$. In Section~\ref{lambdomo}, we introduce the homogeneous analogue $S_{\mathcal{B},\lambda}(\Omega)$ by replacing exponential kernels with $\lambda$-homogeneous kernels. The admissible parameter domain $U_\Omega$ is discussed, which ensures that 
\begin{align*}
    \Re( -\langle x,z\rangle)>0\quad \forall x\in \Omega,
\end{align*}
so that (\ref{newexp}) is well defined on $\Omega$. This leads naturally to the assumption that $\Omega$ is an open pointed cone. We also prove that $S_{\mathcal{B},\lambda}(\Omega)\subset \operatorname{ker}(P(\partial_x))\cap \operatorname{ker}(E^x_\lambda)$. Section~\ref{EuM} studies the Euler operator $E_\lambda^z$, the Mellin-type operator $M_\lambda^z$, and the Paley--Wiener estimates needed in the proof of the main theorem. Section~\ref{lastly} introduces the visibility assumption, which ensures that the visible parts $V_{\mathfrak{p}_\sigma}\cap U_\Omega$ are large enough for the duality argument. Under these assumptions, the main result shows that \(S_{\mathcal B,\lambda}(\Omega)\) is dense in \(\operatorname{ker}(P(\partial_x))\cap \operatorname{ker}(E^x_\lambda)\), thereby giving a \(\lambda\)-homogeneous analogue of the weak fundamental principle. The final section discusses current limitations and possible extensions, including ideas for $\lambda\in \mathbb{N}_0$, non-pointed cones $\Omega$, and a graded family of Noetherian operators. It also formulates a conjectural \(\lambda\)-homogeneous analogue of the strong fundamental principle.

\section{Preliminaries}\label{review}
\subsection{Regularized Mellin formula}
Formula \eqref{intro-mellin} is the Mellin representation of $\beta^\lambda = \exp(\lambda \log(\beta))$, where \(\log\) denotes the principal branch on \(\mathbb C\setminus(-\infty,0]\). Note that throughout the document we always use the principal branch, that is, $\arg (z)\in (-\pi,\pi]$. For \(\Re\lambda<0\) and \(\Re\beta>0\), the integral in \eqref{intro-mellin} converges absolutely. The restriction \(\Re\lambda<0\) is imposed by the behavior of the integrand at \(t=0\). We now record the standard regularized form of this identity, obtained by subtracting finitely many Taylor terms at \(t=0\); compare \cite{ZagierAppendix}. 

First assume \(\Re \lambda<0\) and let \(q\in \mathbb{N}_0\). Splitting the integral in \eqref{intro-mellin} at \(t=1\) and adding and subtracting the Taylor polynomial of degree \(q\) on \((0,1)\), we obtain
\begin{align*}
\Gamma(-\lambda) \beta^\lambda =\int_0^1  \sum_{k=0}^{q}\frac{(-t\beta)^k}{k!} t^{-\lambda-1}\,dt + \int_0^1  \left(e^{-t\beta}- \sum_{k=0}^{q}\frac{(-t\beta)^k}{k!}\right) t^{-\lambda-1}\,dt + \int_1^\infty e^{-t\beta} t^{-\lambda-1}\,dt.
\end{align*}
Here, all three integrals are absolutely convergent under the present assumption $\Re \lambda<0$. Evaluating the first integral term by term gives
\begin{align}\label{mellitrafo}
    \Gamma(-\lambda) \beta^\lambda=\sum_{k=0}^q \frac{(-\beta)^k}{k!(k-\lambda)} + \int_0^1\left(e^{-t\beta}- \sum_{k=0}^{q}\frac{(-t\beta)^k}{k!}\right)
    t^{-\lambda-1}\,dt +\int_1^\infty e^{-t\beta}t^{-\lambda-1} dt .
\end{align}
We shall use (\ref{mellitrafo}) as the regularized Mellin formula. Indeed, due to \(e^{-t\beta}-\sum_{k=0}^q \frac{(-t\beta)^k}{k!}=O(t^{q+1})\), the second integral in (\ref{mellitrafo}) converges near $t=0$ whenever $q>\Re \lambda-1$. The last integral converges for any $\lambda\in \mathbb{C}$ since $\Re \beta>0$. Consequently, for fixed $q\in \mathbb{N}_0$, the right-hand side of (\ref{mellitrafo}) provides a meromorphic extension in $\lambda$ on the half-plane $\Re \lambda<q+1$ with simple poles at $\lambda\in \{0,\dots,q\}$. By choosing $q\in \mathbb{N}_0$ sufficiently large, this formula applies to any $\lambda\in \mathbb{C}\setminus \mathbb{N}_0$. The exclusion of \(\lambda\in\mathbb N_0\) reflects the poles in the formula, and therefore, we mostly assume \(\lambda\in\mathbb C\setminus\mathbb N_0\). The case \(\lambda\in\mathbb N_0\) will be discussed briefly in Section~\ref{choicelambda}.

\subsection{Distributions and Paley--Wiener spaces}
We recall some standard facts; see, for example, \cite{hörmander}. For $A\subset \mathbb{C}^n$, let $\overline{A}$ and $\operatorname{conv}(A)$ denote the closure and convex hull of $A$, respectively. Here and throughout, for multi-indices $\alpha=(\alpha_1,\dots,\alpha_n)\in\mathbb{N}_0^n$, we write
\[
|\alpha|=\alpha_1+\cdots+\alpha_n,
\quad
\partial_z^\alpha=\partial_{z_1}^{\alpha_1}\cdots \partial_{z_n}^{\alpha_n},\quad z^\alpha={z_1}^{\alpha_1}\cdots {z_n}^{\alpha_n}\quad \text{for }z\in \mathbb{C}^n.
\]
Let $\Omega\subset \mathbb{R}^n$ for $n\in \mathbb{N}$ be an open convex set. We denote by $C^\infty(\Omega)$ the space of smooth functions on $\Omega$, endowed with its usual Fréchet topology, and by $\mathcal{E}'(\Omega)$ the space of compactly supported distributions on $\Omega$. If $T:C^\infty(\Omega)\to C^\infty(\Omega) $ is linear and continuous, we write
\begin{align*}
    T^t:\mathcal{E}'(\Omega)\to \mathcal{E}'(\Omega)
\end{align*}
for the transpose operator. For a compact convex set $K\subset \mathbb{R}^n$ and $N\in \mathbb{N}_0$, let $\mathcal{E}'(K,N)$ denote the distributions of order at most $N$ supported in $K$, and set
\begin{align*}
    \mathcal{E}'(K)\coloneqq\bigcup_{N\in \mathbb{N}_0} \mathcal{E}'(K,N).
\end{align*}
For $U\subset \mathbb{C}^n$ open, let $\mathcal{O}(U)$ denote the space of holomorphic functions on $U$. We use the following convention to define the Fourier--Laplace transform
\begin{align*}
    \hat{\nu}(z)\coloneqq\nu\left(e^{\langle x,z\rangle}\right)\quad \text{for }\nu\in \mathcal{E}'(\Omega),~ z\in \mathbb{C}^n,
\end{align*}
and we have $\hat{\nu}\in \mathcal{O}(\mathbb{C}^n)$. In fact, if $K\subset \mathbb{R}^n$ is compact convex and $N\in \mathbb{N}_0$, then by the Paley--Wiener--Schwartz theorem \cite[Thm.~7.3.1]{hörmander} the Fourier--Laplace transform identifies $\mathcal{E}'(K,N)$ isomorphically with the space $\operatorname{PW}(K,N)$, where
\begin{align*}
    \operatorname{PW}(K,N)\coloneqq\{f\in \mathcal{O}(\mathbb{C}^n): \sup_{z\in \mathbb{C}^n} |f(z)| w_{K,N}(z)<\infty\},\quad  w_{K,N}(z)\coloneqq(1+|z|)^{-N} e^{-H_K(\Re z)},
\end{align*}
and $H_K(w)\coloneqq \sup_{x\in K}\langle x,w\rangle$ for $w\in \mathbb{R}^n$ denotes the support function of \(K\). We also set
\begin{align*}
    \operatorname{PW}(K)\coloneqq\bigcup_{N\in \mathbb{N}_0} \operatorname{PW}(K,N),\quad  \operatorname{PW}(\Omega )\coloneqq\bigcup_{\substack{K\Subset \Omega\\ K\text{ compact convex}}} \operatorname{PW}(K) .
\end{align*}

A basic fact that we will use repeatedly is that Paley--Wiener spaces are stable under differential operators with polynomial coefficients: if \(K\subset \mathbb{R}^n\) is compact convex and
\begin{align}\label{laberer}
    \mathcal{B}(z,\partial_z) \coloneqq \sum_{|\alpha|\leq A} a_\alpha(z) \partial^\alpha_z \text{ with }a_\alpha\in \mathbb{C}[z]\quad \text{then} \quad f\in \operatorname{PW}(K)\implies \mathcal{B}(z,\partial_z) f\in \operatorname{PW}(K).
\end{align}
Here, $\mathbb{C}[z]=\mathbb{C}[z_1,\dots,z_n]$ denotes the ring of complex polynomials in $n$ variables. For vector-valued spaces we use the notation 
\[
C^\infty(\Omega)^k\coloneqq\underbrace{C^\infty(\Omega)\times\cdots\times C^\infty(\Omega)}_{k\text{ times}}\quad \text{for } k\in \mathbb{N},
\]
and similarly for $\mathcal{E}'(\Omega)^k$, $\mathcal{O}(\mathbb{C}^n)^k$, $\operatorname{PW}(\Omega)^k$, and $\mathbb{C}[z]^k$.

\subsection{Constant-coefficient systems and the weak fundamental principle}
This subsection reviews material from \cite[Section~7.6]{hörmandercomplex}. We consider a matrix of complex polynomials
\begin{align*}
    P(z) = (P_{jl}(z))_{1\leq j\leq J,~1\leq l\leq L},\quad \text{where }P_{jl}(z)\in \mathbb{C}[z]
\end{align*}
which defines a system of constant-coefficient differential operators
\begin{align*}
    P(\partial_x):C^\infty(\Omega)^L \to C^\infty(\Omega)^J,\quad  (P(\partial_x) u)_j = \sum_{l=1}^L P_{jl}(\partial_x)u_l.
\end{align*}
We are interested in the solution space $S(\Omega)\coloneqq\operatorname{ker}(P(\partial_x))$ which is a closed subspace of $C^\infty(\Omega)^L$ and inherits its Fréchet topology. An element $u\in S(\Omega)$ is called an exponential solution if it admits a representation of the form
\begin{align*}
    \Omega \ni x\mapsto u(x) = p(x) e^{\langle x,z\rangle}\quad \text{for some }z\in \mathbb{C}^n,~p\in \mathbb{C}[x]^L.
\end{align*}
We denote by $ S_e(\Omega)\subset S(\Omega)$ the set of all exponential solutions. The following weak version of the fundamental principle is classical:
\begin{theorem}[\cite{hörmandercomplex}, Thm.~7.6.14]\label{malgrangei}
    Let $\Omega\subset \mathbb{R}^n$ be an open convex set. Then $\operatorname{span}(S_e(\Omega))$ is dense in $S(\Omega)$.
\end{theorem}
A key ingredient is the following factorization result in Paley--Wiener spaces:
\begin{lemma}\label{natos}
    $\operatorname{PW}(\Omega)^L \cap P^t \mathcal{O}(\mathbb{C}^n)^J = P^t \operatorname{PW}(\Omega)^J$
\end{lemma}
which is shown in course of the proof for \cite[Thm.~7.6.14]{hörmandercomplex}. Here, $P^t f\in \mathcal{O}(\mathbb{C}^n)^L$ denotes pointwise multiplication by the transpose matrix $P^t(z)$ for $f\in \mathcal{O}(\mathbb{C}^n)^J$. The inclusion $\supset$ is immediate from (\ref{laberer}); the reverse inclusion is the nontrivial part.

\subsection{Noetherian operators and the Ehrenpreis--Palamodov framework}
The strong form of the fundamental principle requires a more refined analysis of the module generated by $P^t$; this is essentially the content of Section 7.7 in \cite{hörmandercomplex}. Define the $\mathbb{C}[z]$-submodule
\begin{align*}
    M \coloneqq P^t \mathbb{C}[z]^J \subset \mathbb{C}[z]^L
\end{align*}
and choose a primary decomposition
\begin{align}\label{primdeco}
    M= \bigcap_{\sigma=1}^\mu M_{\sigma}
\end{align}
where each $M_\sigma\subset \mathbb{C}[z]^L$ is a primary submodule with associated prime ideal $\mathfrak{p}_\sigma\subset \mathbb{C}[z]$. We denote by
\begin{align*}
    V_{\mathfrak{p}_\sigma}=\{z\in \mathbb{C}^n: \varphi(z)=0\quad \forall \varphi\in \mathfrak{p}_\sigma\}
\end{align*}
the affine variety defined by $\mathfrak{p}_\sigma$. Associated with the decomposition, one can choose finitely many differential operators $\{ \tilde{\mathcal{B}}_{\sigma i}: \mathcal{O}(\mathbb{C}^n)^L\to \mathcal{O}(\mathbb{C}^n): 1\leq \sigma \leq \mu,\,1\leq i \leq \rho_\sigma\}$, called Noetherian operators, of the following form \cite[Thm.~7.7.6]{hörmandercomplex}:
\begin{align}\label{kermer}
\mathcal{O}(\mathbb{C}^n)^L\ni f\mapsto \tilde{\mathcal{B}}_{\sigma i}(z, \partial_z) f= \sum_{|\alpha|\leq A}\sum_{l=1}^L a_{\sigma i,\alpha}^l(z)  \partial_z^\alpha f^l\quad \text{where }a_{\sigma i,\alpha} \in \mathbb{C}[z]^L,
\end{align}
such that the following holds.

\begin{theorem}[\cite{hörmander}, Thm.~7.7.10]\label{makka}
Let $U \subset \mathbb{C}^n$ be a domain of holomorphy. For $f\in \mathcal{O}(U)^L $ we have $f\in P^t \mathcal{O}(U)^J$ if and only if
    \begin{align*}
        \forall \sigma=1,\dots, \mu: \quad \tilde{\mathcal{B}}_{\sigma i}(z,\partial_z) f(z)=0\quad \text{ for all }z\in V_{\mathfrak{p}_\sigma}\cap U \text{ and } i =1,\dots,\rho_\sigma.
    \end{align*}
\end{theorem}

\begin{theorem}\label{useituseit}
    Let $f\in \mathcal{O}(\mathbb{C}^n)^L$ and assume there exists a compact convex set $K\subset \mathbb{R}^n$ and $N\in \mathbb{N}_0$ such that
    \begin{align*} 
    \max_{1\leq \sigma \leq \mu}\max_{1\leq i\leq \rho_\sigma} \sup_{z\in V_{\mathfrak{p}_\sigma}} |\tilde{\mathcal{B}}_{\sigma i}(z,\partial_z) f (z)| w_{K,N}(z)<\infty .
    \end{align*}
    Then $ f\in \operatorname{PW}(K)^L + P^t \mathcal{O}(\mathbb{C}^n)^J$. 
\end{theorem}

In practice, Theorem~\ref{makka} reduces membership in \(P^t\mathcal O(U)^J\) to finitely many vanishing conditions of the Noetherian operators along the varieties \(V_{\mathfrak p_\sigma}\), while Theorem~\ref{useituseit} allows one to deduce Paley--Wiener type growth for \(f\), modulo \(P^t\mathcal O(\mathbb C^n)^J\). Theorem~\ref{useituseit} is the standard application of \cite[Thm.~7.7.15]{hörmandercomplex}, obtained by choosing the plurisubharmonic weights in terms of \(H_K(\Re z)\).

To every scalar-valued differential operator $\tilde{\mathcal{B}}_{\sigma i}$ in (\ref{kermer}), we associate the $\mathbb{C}^L$-valued differential operator
\begin{align}\label{kermer2}
\mathcal{O}(\mathbb{C}^n)\ni f\mapsto \mathcal{B}_{\sigma i}(z, \partial_z) f= \sum_{|\alpha|\leq A} a_{\sigma i,\alpha}(z)  \partial_z^\alpha f \in \mathcal{O}(\mathbb{C}^n)^L.
\end{align}
Moreover, replacing $\partial_z^\alpha$ with $x^\alpha$ in (\ref{kermer2}), we have 
\begin{align*}
    \mathcal{B}_{\sigma i}(z,x) = \sum_{|\alpha|\leq A} a_{\sigma i,\alpha}(z)  x^\alpha\in \mathbb{C}[z,x]^L\quad \text{such that}\quad \mathcal{B}_{\sigma i}(z,\partial_z) e^{\langle x,z\rangle} = \mathcal{B}_{\sigma i}(z,x) e^{\langle x,z\rangle}.
\end{align*}
From now on, we fix a primary decomposition \eqref{primdeco} and an associated family of Noetherian operators. These will be used to generate a family of exponential solutions. For this, define
\begin{align}\label{basis1}
    S_{\mathcal{B}}(\Omega)\coloneqq \{\Omega \ni x\mapsto \mathcal{B}_{\sigma i}(z,x) e^{\langle x,z\rangle}: \sigma=1,\dots,\mu;~i=1,\dots, \rho_\sigma,~z\in V_{\mathfrak{p}_\sigma}\}\subset C^\infty(\Omega)^L.
\end{align}
By the discussion in \cite[p.~245]{hörmandercomplex}, we have
\begin{align*}
    S_{\mathcal{B}}(\Omega) \subset S_e(\Omega)\subset S(\Omega).
\end{align*}
Also, for every $\nu\in \mathcal{E}'(\Omega)^L$, differentiation under the pairing defining the Fourier--Laplace transform yields
\begin{align}\label{kllae}
    \tilde{\mathcal{B}}_{\sigma i}(z,\partial_z) \hat{\nu}(z) = \nu( \mathcal{B}_{\sigma i}(z,\partial_z) e^{\langle x,z\rangle}) = \nu( \mathcal{B}_{\sigma i}(z,x) e^{\langle x,z\rangle}).
\end{align}

As a direct consequence of the preceding results, one obtains a refinement of Theorem \ref{malgrangei}:
\begin{theorem}\label{wotuw}
    Let $\Omega\subset \mathbb{R}^n$ be an open convex set. Then $\operatorname{span}(S_\mathcal{B}(\Omega))$ is dense in $S(\Omega)$.
\end{theorem}
\begin{proof}
    Let $\nu\in \mathcal{E}'(\Omega)^L$ vanish on $S_\mathcal{B}(\Omega)$. By (\ref{kllae}) and Theorem \ref{makka} we have $\hat{\nu}\in \operatorname{PW}(\Omega)^L\cap P^t \mathcal{O}(\mathbb{C}^n)^J$. Lemma \ref{natos} yields $\hat{\nu} = P^t \hat{\xi} $ for some $\xi \in \mathcal{E}'(\Omega)^J$. Therefore, for every $\varphi\in S(\Omega)$,
    \begin{align*}
        \langle \nu, \varphi\rangle = \langle P(\partial_x)^t\xi , \varphi\rangle = \langle \xi, P(\partial_x) \varphi\rangle=0.
    \end{align*}
    Thus $\nu$ vanishes on $S(\Omega)$, and the claim follows by an application of Hahn-Banach.
\end{proof}
Theorem~\ref{wotuw} remains a weak version of the fundamental principle, since it asserts only density, but it sharpens Theorem~\ref{malgrangei} by showing that density already holds for the distinguished family \(S_{\mathcal B}(\Omega)\). Our main result, Theorem~\ref{main}, will be a $\lambda$-homogeneous analogue of Theorem~\ref{wotuw}.

By contrast, the strong form of the fundamental principle, namely the Ehrenpreis--Palamodov representation theorem, gives more: every solution in $S(\Omega)$ admits a representation of the form
\begin{align*}
    u(x) = \sum^\mu_{\sigma=1} \sum^{\rho_\sigma}_{i=1} \int_{V_{\mathfrak{p}_\sigma}} \mathcal{B}_{\sigma i}(z,x) e^{\langle x,z\rangle} d\omega_{\sigma i}(z) .
\end{align*}
Here, each $\omega_{\sigma i}$ is a suitable weighted complex Radon measure on $V_{\mathfrak{p}_\sigma}$ such that the integrals converge absolutely; see \cite[Thm.~7.7.16]{hörmandercomplex}, the following remark therein, and also \cite{Hansen1983}.

\section{$\lambda$-homogeneous solutions}\label{lambdomo}
This section introduces the natural \(\lambda\)-homogeneous analogue of the exponential family \(S_{\mathcal B}(\Omega)\).

\subsection{Homogeneous solution space}
From now on we assume that $P\in \mathbb{C}[z]^{J\times L}$ is homogeneous of degree $m\in \mathbb{N}$, that is, each nonzero entry of the matrix $P(z)$ is a homogeneous polynomial of degree $m$.
In this case, it is well-known that the associated prime ideals $\mathfrak{p}_\sigma \subset \mathbb{C}[z]$ are homogeneous; see, for example, \cite[Sec.~3.5]{E1995}. In particular, each variety $V_{\mathfrak{p}_\sigma}$ is a complex cone, i.e. $\zeta \,V_{\mathfrak{p}_\sigma}\subset V_{\mathfrak{p}_\sigma}$ for all $\zeta\in \mathbb{C}$. The homogeneity of $P$ is compatible with the Euler operator $E_\lambda^x = x\cdot \nabla_x -\lambda$, where $\lambda\in \mathbb{C}$, in the sense that
\begin{align}\label{ssos}
    P(\partial_x) E^x_\lambda  = E^x_{\lambda-m} P(\partial_x)
\end{align}
which follows from $\partial_x^\alpha( x\cdot \nabla_x) = (x\cdot \nabla_x + |\alpha|)\partial_x^\alpha $. Here and in the sequel, scalar operators such as \(E_\lambda^x\) are understood to act componentwise on vector-valued spaces. For an open set $\Omega\subset \mathbb{R}^n$, we define the space of $\lambda$-homogeneous functions by
\begin{align*}
    H_\lambda(\Omega)\coloneqq\operatorname{ker}(E^x_\lambda:C^\infty(\Omega)\to C^\infty(\Omega)),
\end{align*}
so by (\ref{ssos}) we get
\begin{align*}
    P(\partial_x):H_\lambda(\Omega)^L\to H_{\lambda-m}(\Omega)^J.
\end{align*}

\subsection{Geometry of $\Omega$}\label{sec:geo}
In the following, we study the space of $\lambda$-homogeneous solutions
\begin{align*}
    S(\Omega)\cap H_\lambda(\Omega)^L.
\end{align*}
Since both $S(\Omega) = \operatorname{ker}( P(\partial_x))$ and $H_\lambda(\Omega)^L=\operatorname{ker}(E_\lambda^x)^L$ are closed subspaces of $C^\infty(\Omega)^L$, their intersection is again a closed subspace. We endow $S(\Omega)\cap H_\lambda(\Omega)^L$ with the induced Fréchet topology. As in the previous section, $\Omega\subset \mathbb{R}^n$ will be open and convex. Since homogeneity is naturally tied to dilation-invariant sets, it is natural to assume that $\Omega\subset \mathbb{R}^n$ is a cone, that is, $tx\in \Omega$ for $x\in \Omega$ and $t>0$.

\begin{remark}
Let \(\Sigma\coloneqq \Omega\cap \mathbb{S}^{n-1}\). The restriction map $H_\lambda(\Omega)\ni u\mapsto u|_\Sigma \in C^\infty(\Sigma)$ is an isomorphism with inverse given by
\[
C^\infty(\Sigma)\ni f\mapsto \bigl[x\mapsto |x|^\lambda f(x/|x|)\bigr]\in H_\lambda(\Omega).
\]
In particular, a \(\lambda\)-homogeneous function on \(\Omega\) is uniquely determined by its values on \(\Sigma\).
\end{remark}

If $0\in \Omega$, then $\Omega=\mathbb{R}^n$ for $\Omega$ an open convex cone. In that case, smoothness at the origin is very restrictive: a smooth $\lambda$-homogeneous function on $\mathbb{R}^n$ must be a polynomial. Consequently, for $\lambda\in\mathbb{N}_0$, the space $H_\lambda(\mathbb{R}^n)$ is a subspace of the finite-dimensional space of homogeneous polynomials of degree $\lambda$, and for $\lambda \in  \mathbb{C}\setminus \mathbb{N}_0$, we have $H_\lambda(\mathbb{R}^n) = \{0\}$. This observation illustrates that the genuinely interesting situation arises when $0\notin \Omega$. For reasons that become clear in a moment, we strengthen this slightly.
\begin{definition}
   A cone $\Omega\subset \mathbb{R}^n$ is called pointed if $\overline{\Omega} \cap (-\overline{\Omega}) = \{0\}$.
\end{definition}

Equivalently, $\Omega$ is pointed if and only if $\overline{\Omega}$ contains no line through the origin. In particular, if $\Omega$ is open and pointed, then $0\notin \Omega$. In the following, we say that $\Omega\subset \mathbb{R}^n$ is OPC if it is a nonempty open pointed convex cone.

\begin{example}\label{halfspace}
   The half-space $\Omega = \mathbb{R}^{n-1}\times \mathbb{R}_{>0}$ is an open convex cone and satisfies $0\notin \Omega$. However, it is not pointed since $\overline{\Omega} \cap (-\overline{\Omega}) = \mathbb{R}^{n-1}\times \{0\}$.
\end{example}

\subsection{Dual cones}
The exponentials $x\mapsto e^{\langle x,z\rangle}$ with parameter $z\in \mathbb{C}^n$ are the natural building blocks for the classical solution space $S(\Omega)$; see (\ref{basis1}). In the homogeneous setting, we seek analogous building blocks for $S(\Omega)\cap H_\lambda(\Omega)^L$. The guiding principle is that the Mellin transform converts exponentials into $\lambda$-homogeneous kernels. Taking $\beta=-\langle x,z\rangle$ in \eqref{mellitrafo}, the factor $e^{-t\beta}$ becomes $e^{t\langle x,z\rangle}$, which will be compatible with the exponentials in (\ref{basis1}). The condition $\Re \beta>0$ then reads $\langle x,\Re z\rangle<0$ for all $x\in \Omega$. This motivates introducing the following: For any nonempty subset $X\subset \mathbb{R}^n$, define the strict negative dual cone
\begin{align*}
    X^\vee \coloneqq\{\eta\in \mathbb{R}^n: \langle x,\eta\rangle <0\quad \forall x\in \overline{X}\setminus \{0\}\}\quad \text{and}\quad U_X\coloneqq \{z\in \mathbb{C}^n: \Re z\in X^\vee \}.
\end{align*}
The fundamental building blocks for $\lambda$-homogeneous functions will be
\begin{align}\label{bblock}
    \Omega \ni x\mapsto (-\langle x,z\rangle)^\lambda\quad \text{for }z\in U_\Omega.
\end{align}
Since \(\Re(-\langle x,z\rangle)>0\) for \(x\in \Omega\) and \(z\in U_\Omega\), this defines a smooth function on \(\Omega\), avoiding the branch cut $(-\infty,0]$. We record some geometric properties.

\begin{lemma}\label{geostruc}
Let \(\Omega\subset\mathbb R^n\) be OPC, and let
\(K\subset\mathbb R^n\) be compact and convex with \(0\notin K\). Then:
\begin{enumerate}
    \item \(U_\Omega\) and \(U_K\) are nonempty open real cones in \(\mathbb C^n\).
    \item We have $U_K=\{z\in\mathbb C^n:H_K(\Re z)<0\}$.
\end{enumerate}
\end{lemma}
\begin{proof}
We first prove the identity for \(U_K\). Since suprema are attained on compacta, we have
\begin{align*}
    z\in U_K \iff \forall x\in K: \langle x, \Re z\rangle<0  \iff   \sup_{x\in K}\langle x,\Re z\rangle  < 0 \iff  H_K(\Re z) <0.
\end{align*}
This proves (2). In particular, \(U_K\) is open, because \(z\mapsto H_K(\Re z)\) is continuous.
It is a real cone by the homogeneity of the scalar product. It remains to prove nonemptiness. Since $K$ is compact, convex, and disjoint from $0$, strict separation gives $\eta\in \mathbb{R}^n$ and $c>0$ such that $\langle x,\eta\rangle<-c$ for all $x\in K$. Thus, $\eta\in K^\vee\subset U_K\neq \emptyset$. For $\Omega$, set $K_\Omega \coloneqq \operatorname{conv}(\overline{\Omega} \cap \mathbb{S}^{n-1})$. This set is compact and convex. Moreover, $0\notin K_\Omega$, otherwise pointedness of $\Omega$ would fail. Since $\Omega$ is a cone, we have $U_{K_\Omega} = U_\Omega$. Applying the preceding part to $K_\Omega$ completes the proof.
\end{proof}

Pointedness ensures nonemptiness in Lemma~\ref{geostruc}. By contrast, for the half-space in Example~\ref{halfspace}, one has \(\Omega^\vee=\emptyset\), hence \(U_\Omega=\emptyset\), and therefore the family of kernels in (\ref{bblock}) is empty.

\subsection{The $\lambda$-homogeneous solution family $S_{\mathcal{B},\lambda}(\Omega)$}\label{dafamilia}
Let $\Omega\subset \mathbb{R}^n$ be OPC. Define the following solution family
\begin{align*}
    S_{\mathcal{B},\lambda}(\Omega)\coloneqq \left\{ \Omega \ni x\mapsto \mathcal{B}_{\sigma i}(z,\partial_z) \left(-\langle x,z\rangle\right)^\lambda: \sigma=1,\dots,\mu;~i=1,\dots, \rho_\sigma,~z\in V_{\mathfrak{p}_\sigma}\cap U_\Omega \right\}.
\end{align*}
Compared to \eqref{basis1}, the exponential factor \(e^{\langle x,z\rangle}\) is replaced by the \(\lambda\)-homogeneous kernel \(\left(-\langle x,z\rangle\right)^\lambda\), and the parameter set is restricted to \(V_{\mathfrak{p}_\sigma}\cap U_\Omega\) so that the principal branch is well defined on \(\Omega\). It is possible that \(V_{\mathfrak p_\sigma}\cap U_\Omega\) is empty for some \(1\leq \sigma\leq \mu\). In that case the corresponding component contributes
no kernels to \(S_{\mathcal B,\lambda}(\Omega)\). Whether such components are irrelevant depends on their position relative to \(U_\Omega\), a point taken up in Section~\ref{sec:invis}.

The main goal of this work is to show that $ S_{\mathcal{B},\lambda}(\Omega)$ is dense in $S(\Omega) \cap H_\lambda(\Omega)^L$. Writing the elements in $S_{\mathcal{B},\lambda}(\Omega)$ more explicitly, we have
    \begin{align*}
         \mathcal{B}_{\sigma i}(z,\partial_z) \left(-\langle x,z\rangle\right)^\lambda = \sum_{|\alpha|\leq A} (-1)^{|\alpha|} (\lambda)_{|\alpha|}a_{\sigma i,\alpha}(z)  x^\alpha \left(-\langle x,z\rangle\right)^{\lambda-|\alpha|},
    \end{align*}
where $(\lambda)_{|\alpha|}\coloneqq \lambda(\lambda-1)\dots(\lambda-|\alpha|+1)$. In particular, the elements of \(S_{\mathcal B,\lambda}(\Omega)\) are finite linear combinations of terms
\(x^\alpha \left(-\langle x,z\rangle\right)^{\lambda-|\alpha|}\) with coefficients depending polynomially on \(z\).

\begin{proposition}\label{schaumal}
    Let $\lambda \in \mathbb{C}$ and $\Omega$ be OPC. Then $S_{\mathcal{B},\lambda}(\Omega)\subset S(\Omega) \cap H_\lambda(\Omega)^L$.
\end{proposition}

\begin{proof}
Fix $1\leq \sigma \leq \mu$, $1\leq i\leq \rho_\sigma$ and $z\in V_{\mathfrak{p}_\sigma}\cap U_\Omega$, and set \( u(x)\coloneqq   \mathcal{B}_{\sigma i}(z,\partial_z) \left(-\langle x,z\rangle\right)^\lambda\). Note 
\begin{align*}
    (x\cdot \nabla_x) \left(-\langle x,z\rangle\right)^\lambda = \lambda  \left(-\langle x,z\rangle\right)^{\lambda}.
\end{align*}
Thus $E_\lambda^x \left(-\langle x,z\rangle\right)^\lambda=0$. Moreover, \(E_\lambda^x\) and \(\mathcal{B}_{\sigma i}(z,\partial_z)\) commute, since they act in different variables. Hence \(u\in H_\lambda(\Omega)^L\) since
    \begin{align*}
        E_\lambda^x u(x) = E_\lambda^x \mathcal{B}_{\sigma i}(z,\partial_z)  \left(-\langle x,z\rangle\right)^\lambda =  \mathcal{B}_{\sigma i}(z,\partial_z) E_\lambda^x  \left(-\langle x,z\rangle\right)^\lambda = 0.
    \end{align*}
To prove $u\in S(\Omega)$, fix $1\leq j\leq J$ and $x\in \Omega$. Calculate
    \begin{align}\nonumber
        &(P(\partial_x)  u)_j (x) = \sum_{l=1}^L P_{jl}(\partial_x)\mathcal{B}^l_{\sigma i}(z,\partial_z)   \left(-\langle x,z\rangle\right)^\lambda =  \sum_{l=1}^L \mathcal{B}^l_{\sigma i}(z,\partial_z)  P_{jl}(\partial_x) \left(-\langle x,z\rangle\right)^\lambda \\ 
        =& (-1)^m (\lambda)_m \sum_{l=1}^L \mathcal{B}^l_{\sigma i}(z,\partial_z)  \bigl( P_{jl}(z) \left(-\langle x,z\rangle\right)^{\lambda-m}\bigr) = (-1)^m (\lambda)_m \tilde{\mathcal B}_{\sigma i}(w,\partial_w)\bigl(P^t(w)g_{j,x}(w)\bigr)\Big|_{w=z}\label{konumber} 
    \end{align}
where we use that $P$ is homogeneous of degree $m$ and set $ g_{j,x}(w):=\left(-\langle x,w\rangle\right)^{\lambda-m} e_j$ for \(e_j\) the \(j\)-th standard basis vector in $\mathbb{C}^J$. Note that $g_{j,x}$ is well-defined and holomorphic for a small ball \(w\in W\Subset U_\Omega\) centered at \(z\). Since \(W\) is convex, it is a domain of holomorphy and (\ref{konumber}) vanishes by Theorem~\ref{makka}. As \(j\) and \(x\) were arbitrary, \(u\in S(\Omega)\).
\end{proof}

\section{Euler and Mellin-type operators}\label{EuM}
The purpose of this section is threefold. First, we describe the Fourier--Laplace transformed Euler operator $E_\lambda^z$ and invert it on \(\mathcal{O}(\mathbb{C}^n)\) for \(\lambda \notin \mathbb{N}_0\). Second, for \(K\Subset \Omega\) a compact convex set, we introduce the Mellin-type operator \(M_\lambda^z\) on the Paley--Wiener space \(\operatorname{PW}(K)\) which satisfies \(E_\lambda^z M_\lambda^z f = 0\) and sends the exponential kernel \(e^{\langle x,z\rangle}\) to the homogeneous kernel \((-\langle x,z\rangle)^\lambda\). Third, we prove Paley--Wiener estimates needed for the main result.

\subsection{The Euler operator on the Fourier--Laplace side}
We retain the notation $E_\lambda^z \coloneqq z\cdot \nabla_z-\lambda$ for $z\in \mathbb{C}^n$. Passing to the Fourier--Laplace transform, the transpose of the Euler operator becomes the Euler operator in the \(z\)-variables:
\begin{align*}
    \widehat{(E^x_\lambda)^t \nu }(z) = \nu\left(E^x_\lambda e^{\langle x,z\rangle}\right) = \nu\left(( x\cdot z - \lambda) e^{\langle x,z\rangle}\right) = \nu\left(( z\cdot \nabla_z  - \lambda) e^{\langle x,z\rangle}\right) =E^z_{\lambda} \hat{\nu}(z),\quad \nu\in \mathcal{E}'(\Omega).
\end{align*}
Since \(P^t(z)\) is homogeneous of degree \(m\), we have
\begin{align}\label{zweo}
   E^z_\lambda(P^t g)=P^t E^z_{\lambda-m} g\quad \text{for } g\in \mathcal{O}(\mathbb{C}^n)^J.
\end{align}
Next, we discuss the inverse of $ E^z_\lambda : \mathcal{O}(\mathbb{C}^n)\to \mathcal{O}(\mathbb{C}^n)$. Using the Taylor expansion, we write $\mathcal{O}(\mathbb{C}^n)\ni f = \sum_{k=0}^\infty f_k$ where each $f_k$ is a homogeneous polynomial of degree $k$ such that
\begin{align*}
    E^z_\lambda f = \sum_{k=0}^\infty (k-\lambda) f_k.
\end{align*}
The operator \(E^z_\lambda\) is invertible on \(\mathcal{O}(\mathbb{C}^n)\) precisely when \(\lambda\notin\mathbb{N}_0\). In this case, the inverse is given by
\begin{align}\label{nadelohr}
    (E^z_\lambda)^{-1} f = \sum_{k=0}^\infty \frac{f_k}{k-\lambda},
\end{align}
which is entire and thus $(E^z_\lambda)^{-1}:\mathcal{O}(\mathbb{C}^n)\to \mathcal{O}(\mathbb{C}^n)$ is well defined. Indeed, if the power series \(f(z)=\sum_{\alpha}c_\alpha z^\alpha\) has infinite radius of convergence, then so does $\sum_\alpha \frac{c_\alpha z^\alpha}{|\alpha|-\lambda}$.

Fix \(q\in \mathbb{N}_0\) such that \(q>\Re\lambda-1\). A useful representation of $(E^z_\lambda)^{-1}$ is given by
\begin{align}\label{reppi}
    (E^z_\lambda)^{-1} f(z)= \sum_{k=0}^q \frac{f_k(z)}{k-\lambda} +\int_{0}^{1} t^{-\lambda-1}\left( f(tz)-\sum_{k=0}^q t^{k} f_k(z)\right) dt.
\end{align}
To see this, note that
\[
f(tz)-\sum_{k=0}^q t^k f_k(z)=\sum_{k=q+1}^\infty t^k f_k(z) =O(t^{q+1})\quad \text{for }t\to 0
\]
and the integral in (\ref{reppi}) converges because \(q>\Re\lambda-1\). Integrating the homogeneous expansion term by term gives \eqref{nadelohr}.

\subsection{The operator $M_\lambda^z$}
Let $\lambda\in \mathbb{C}\setminus \mathbb{N}_0$ and $0\notin K\Subset \Omega$ be a compact convex set. Motivated by \eqref{mellitrafo} and \ref{reppi}, we define
\begin{align}\label{m2}
 M^z_\lambda :& \operatorname{PW}(K)\to \mathcal{O}(U_K),\quad M^z_\lambda  f(z)\coloneqq (E^z_\lambda)^{-1} f(z) + \int_{1}^{\infty} t^{-\lambda-1} f(tz) dt,\quad z\in U_K,
\end{align}
and briefly verify that $M_\lambda^z$ is well defined. Since $(E^z_\lambda)^{-1} f$ is entire, only the integral term requires justification. For this, consider $f\in \operatorname{PW}(K,N)$, that is,
\begin{align*}
   | f(z)|\leq C (1+|z|)^{N} e^{H_K(\Re z)}\quad \forall z\in \mathbb{C}^n\quad \text{for some }C>0 \text{ and } N\in \mathbb{N}_0.
\end{align*}
For $z\in U_K$ we have
\begin{align*}
   \int_{1}^{\infty} |t^{-\lambda-1} f(tz)| dt \leq C \int_{1}^{\infty} t^{-\Re \lambda-1} (1+t|z|)^{N} e^{ t H_K(\Re z)} 
   \leq  C  (1+|z|)^{N} \int_{1}^{\infty} t^{N-\Re \lambda-1} e^{t H_K(\Re z)}dt,
\end{align*}
where we have used  $(1+t|z|)^{N}\leq  t^{N} (1+|z|)^{N}$ for $t\geq 1$ and homogeneity of $H_K$. Due to $z\in U_K$, $H_K(\Re z)<0$ by Lemma \ref{geostruc} such that the integral in (\ref{m2}) is absolutely convergent, and thus, $M^z_\lambda  f(z)$ is well defined. If \(Q\Subset U_K\) is compact, then $\sup_{z\in Q} H_K(\Re z)<0$, so the same estimate is uniform on \(Q\), and absolute convergence remains true after differentiating under the integral sign, since any derivative $\partial_z^\alpha$ produces only an additional factor \(t^{|\alpha|}\). Hence \(M^z_\lambda f\in \mathcal{O}(U_K)\). Moreover, we can move differential operators under the integral in future computations.

For fixed \(x\in K\), the exponential \(z\mapsto e^{\langle x,z\rangle}\) belongs to \(\operatorname{PW}(K)\). Equivalently, it is the Fourier--Laplace transform of the Dirac mass \(\delta_x\in \mathcal E'(K)\). Hence, \(M_\lambda^z e^{\langle x,z\rangle}\) is well defined for \(z\in U_K\), and \eqref{mellitrafo}, with \(\beta=-\langle x,z\rangle\), yields
\begin{align*}
    M_\lambda^z e^{\langle x,z\rangle} =\Gamma(-\lambda)(-\langle x,z\rangle)^\lambda .
\end{align*}
Consequently, for $\nu\in \mathcal E'(K)$, the following fundamental identity holds, similar to the one in (\ref{kllae}):
\begin{align}
\tilde{\mathcal{B}}_{\sigma i}(z,\partial_z)\, M_\lambda^z \hat{\nu}(z)
= \nu\!\left(\mathcal{B}_{\sigma i}(z,\partial_z)\, M_\lambda^z e^{\langle x,z\rangle}\right) = \Gamma(-\lambda)\, \nu\!\left(\mathcal{B}_{\sigma i}(z,\partial_z)\, (-\langle x,z\rangle)^\lambda\right),
\quad z \in U_K.
\label{mellin-distribution-identity}
\end{align}

\begin{lemma}
Let $\lambda \in \mathbb{C}\setminus \mathbb{N}_0$ and $0\notin K\subset \mathbb{R}^n$ be a compact convex set. Then $E_\lambda^z M_\lambda^z f=0$ on $U_K$ for $f\in \operatorname{PW}(K)$.
\end{lemma}
\begin{proof}
For every $t>0$, the chain rule along rays gives $t\frac{d}{dt}f(tz) = (w\cdot \nabla_w f(w))|_{w=tz}$. Hence, by the product rule
\begin{align*}
    \frac{d}{dt}(t^{-\lambda} f(tz)) = t^{-\lambda-1}(E^z_\lambda f)(tz).
\end{align*}
Now take $z\in U_K$. Differentiating under the integral and using $E_\lambda^z[f(tz)] = (E_\lambda^z f)(tz)$, we obtain
\begin{align}\label{argoo}
     E^z_\lambda M^z_\lambda f(z) = f(z) + \int_1^\infty t^{-\lambda-1} (E^z_\lambda f)(tz) =  f(z) + \int_1^\infty \frac{d}{dt}(t^{-\lambda} f(tz)) dt= \lim_{t\to \infty}t^{-\lambda }f(tz).
\end{align}
It remains to show that the limit in (\ref{argoo}) is zero. Since $f\in \operatorname{PW}(K)$, there is $C>0$ and $N\in \mathbb{N}_0$ such that
\[
|t^{-\lambda}f(tz)|
\le Ct^{-\Re \lambda}(1+t|z|)^N e^{t H_K(\Re z)}
\to 0\quad \text{for }t\to \infty,
\]
where the convergence is due to $H_K(\Re z)<0$ by Lemma \ref{geostruc}.
\end{proof}

\subsection{Paley--Wiener estimates}
Although \(E_\lambda^z\) is invertible on \(\mathcal O(\mathbb C^n)\) for $\lambda\notin \mathbb{N}_0$, the inverse \((E_\lambda^z)^{-1}\) does not preserve \(\operatorname{PW}(K)\) in general. Rather, it produces Paley--Wiener growth associated with \(\operatorname{conv}(K\cup\{0\})\). This is in contrast with \( E_\lambda^z \operatorname{PW}(K)\subset \operatorname{PW}(K)\), which follows from (\ref{laberer}).

\begin{lemma}\label{chito}
Let $\lambda \in \mathbb{C}\setminus \mathbb{N}_0$, and let $K\subset \mathbb{R}^n$ be a compact convex set. Then
\[
(E^z_\lambda)^{-1} \operatorname{PW}(K)\subset \operatorname{PW}(\operatorname{conv}(K\cup \{0\})).
\]
\end{lemma}
\begin{proof}
Take $f\in \operatorname{PW}(K)$, fix $q>\Re \lambda -1$ and consider the representation (\ref{reppi}). Therein, the finite sum is a polynomial, hence belongs to $\operatorname{PW}(\{0\})$. For the integral term, apply Taylor's remainder theorem to the one-variable function $g(s)\coloneqq f(sz)$. Because $g^{(k)}(s) = (z\cdot \nabla_z)^k f(sz)$, we obtain
\begin{align}\label{talya}
    f(tz)-\sum_{k=0}^q t^{k} f_k(z) = \frac{t^{q+1}}{q!}\int_0^1 (1-s)^q (z\cdot \nabla_z)^{q+1}f(stz) ds.
\end{align}
Since $(z\cdot \nabla_z)^{q+1}f\in \operatorname{PW}(K)$ due to (\ref{laberer}), there exists $C>0$ and $N\in \mathbb{N}_0$ such that 
\begin{align}\label{somate}
    |(z\cdot \nabla_z)^{q+1}f(z)|\leq C (1+|z|)^N e^{H_K(\Re z)}\quad \forall z\in \mathbb{C}^n.
\end{align}
Using (\ref{talya}) and (\ref{somate}), we estimate for $z\in \mathbb{C}^n$:
\begin{align*}
&\Big|\int_{0}^{1} t^{-\lambda-1}\Big( f(tz)-\sum_{k=0}^q t^{k} f_k(z)\Big) dt\Big|\leq \int_{0}^{1} \frac{t^{-\Re \lambda+q}}{q!} \int_0^1 (1-s)^q|(z\cdot \nabla_z)^{q+1}f(stz) |ds dt\\
\leq & C \int_{0}^{1} \frac{t^{-\Re \lambda+q}}{q!} \int_0^1 (1+st|z|)^{ N}  e^{st H_K(\Re z)} ds dt \leq C (1+|z|)^{ N} e^{\max(0, H_K(\Re z))} \int_{0}^{1} \frac{t^{-\Re \lambda+q}}{q!} dt \\
= &\frac{C}{q! (q+1-\Re \lambda)} (1+|z|)^{ N} e^{\max(0, H_K(\Re z))},
\end{align*}
where we have used $0\leq s,t\leq 1$ and $e^{st H_K(\Re z)}\leq e^{\max(0, H_K(\Re z))}$ for the third inequality. Since $\max(0,H_K(w))=H_{\operatorname{conv}(K\cup\{0\})}(w)$, we obtain $(E^z_\lambda)^{-1} f\in \operatorname{PW}(\operatorname{conv}(K\cup\{0\}))$.
\end{proof}

\begin{lemma}\label{chito2}
Let $\lambda \in \mathbb{C}\setminus \mathbb{N}_0$, and let $0\notin K\subset \mathbb{R}^n$ be a compact convex set. Consider $f\in \operatorname{PW}(K)^L$, an arbitrary subset $V\subset \mathbb{C}^n$, and let $\tilde{\mathcal{B}}(z, \partial_z)$ be a differential operator of the form (\ref{kermer}):
\begin{align*}
    \tilde{\mathcal{B}}(z, \partial_z) = \sum_{|\alpha|\leq A}a_\alpha(z) \partial_z^\alpha,\quad a_\alpha \in \mathbb{C}[z]^L.
\end{align*}

Assume $\tilde{\mathcal{B}}(z,\partial_z) M^z_\lambda   f(z)=0$ for all $z\in V\cap U_K$. Then there exists $N\in \mathbb{N}_0$ such that
    \begin{align*}
        \sup_{z\in V} \big| \tilde{\mathcal{B}}(z,\partial_z) (E^z_\lambda)^{-1} f(z)\big| w_{K,N}(z)<\infty.
    \end{align*}
\end{lemma}

To motivate this result, consider first the simple case $L=1$, $V=\mathbb{C}^n$ and $\tilde{\mathcal{B}}=1$. Then the statement reduces to the following: if $f\in \operatorname{PW}(K)$ and $M^z_\lambda f=0$, then $(E^z_\lambda)^{-1} f\in \operatorname{PW}(K)$. So the obstruction to invertibility of $(E^z_\lambda)^{-1}$ within $\operatorname{PW}(K)$ is measured by $M^z_\lambda$.

\begin{proof}
Choose constants $C_f, C_\mathcal{B}>0$ and $N_f, N_\mathcal{B}\in \mathbb{N}_0$ such that 
\begin{align*}
    |\partial_z^\alpha f(z)|\leq C_f(1+|z|)^{N_f} e^{H_K(\Re z)},\quad |a_\alpha(z)|\leq C_\mathcal{B}(1+|z|)^{N_\mathcal{B}}\quad \forall |\alpha|\leq A,~z\in \mathbb{C}^n.
\end{align*}
We split the estimate into two regimes:
\begin{align*}
    V = \left(V\cap \{H_K(\Re z)\leq -1\}\right) \cup \left(V\cap \{H_K(\Re z)\geq -1 \}\right).
\end{align*}
First, consider $z\in V\cap \{H_K(\Re z)\leq -1\}$, so $z\in U_K$ by Lemma \ref{geostruc}. By the assumption and the definition of $M^z_\lambda$,
\begin{align*}
   \tilde{\mathcal{B}}(z,\partial_z) (E^z_\lambda)^{-1} f(z) = - \tilde{\mathcal{B}}(z,\partial_z) \int_1^\infty t^{-\lambda-1}f(tz) dt.
\end{align*}
Differentiating under the integral sign yields
    \begin{align*}
        \left| \tilde{\mathcal{B}}(z,\partial_z) (E^z_\lambda)^{-1} f(z)\right| 
        \leq  &\sum_{|\alpha|\leq A} \int_1^\infty \left|  t^{-\lambda-1+|\alpha|}  a_\alpha(z) (\partial^\alpha_z f)(tz) \right|dt\\ 
        \leq  & C_f C_\mathcal{B} \sum_{|\alpha|\leq A} \int_1^\infty   t^{-\Re \lambda-1+|\alpha|}  (1+|z|)^{N_\mathcal{B}} (1+t|z|)^{N_f} e^{tH_K(\Re z)} dt\\ 
        \leq  & C_f C_\mathcal{B} (1+|z|)^{N_\mathcal{B}+N_f} e^{H_K(\Re z)}  \sum_{|\alpha|\leq A}  \int_1^\infty   t^{N_f+|\alpha|-\Re \lambda-1}   e^{(t-1)H_K(\Re z)} dt.
    \end{align*}
    Since $H_K(\Re z)\leq -1$, the last integral is bounded uniformly in \(z\). Hence, there is $C_1>0$ such that
\begin{align}\label{1}
    \left| \tilde{\mathcal{B}}(z,\partial_z) (E^z_\lambda)^{-1} f(z)\right| 
    \le
    C_1(1+|z|)^{N_\mathcal{B}+N_f}e^{H_K(\Re z)}\quad
    \text{for } z\in V\cap \{H_K(\Re z)\leq -1\}.
\end{align}
Next, set $K_0\coloneqq\operatorname{conv}(K\cup \{0\})$. By Lemma \ref{chito} we have $(E^z_\lambda)^{-1} f\in \operatorname{PW}(K_0)$ and by (\ref{laberer}) also $\tilde{\mathcal{B}}(z,\partial_z) (E^z_\lambda)^{-1} f\in \operatorname{PW}(K_0)$. So there is $C_2>0$ and $N_2\in \mathbb{N}_0$ such that
\begin{align*}
    | \tilde{\mathcal{B}}(z,\partial_z) (E^z_\lambda)^{-1} f(z)| \le C_2(1+|z|)^{N_2} e^{\max\{0,H_K(\Re z)\}}\quad \forall z\in \mathbb{C}^n.
\end{align*}
On $\{H_K(\Re z)\geq -1\}$ we have $e^{\max\{0,H_K(\Re z)\}}\leq e e^{H_K(\Re z)}$, so we can estimate
\begin{align}\label{2}
    | \tilde{\mathcal{B}}(z,\partial_z) (E^z_\lambda)^{-1} f(z)|
\le
C_2(1+|z|)^{N_2} e ~e^{H_K(\Re z)} \quad \text{for }z\in \{ H_K(\Re z) \geq -1\}.
\end{align}
Combining (\ref{1}) and (\ref{2}) for $V\cap \{-1\leq H_K(\Re z)\}$, we can choose $N=\max(N_\mathcal{B}+N_f,N_2)$ so
\begin{align*}
    \sup_{z\in V} | \tilde{\mathcal{B}}(z,\partial_z) (E^z_\lambda)^{-1} f(z)| w_{K,N}(z)\leq \max(C_1, C_2 e)<\infty.
\end{align*}
\end{proof}

Lemma \ref{chito2} is the key estimate that converts vanishing of $\tilde{\mathcal{B}}(z,\partial_z) M^z_\lambda   f $ on $V\cap U_K$ into the Paley--Wiener bounds required for Theorem~\ref{useituseit}. In particular, the assertion also covers the case $V\cap U_K=\emptyset$, where the vanishing assumption is vacuous. In this case the first region is empty, since \(\{H_K(\Re z)\le -1\}\subset U_K\).

\section{The main result}\label{lastly}

\subsection{Invisible components}\label{sec:invis}

The proof of the main theorem will follow the same duality strategy as Theorem~\ref{wotuw}, but in the \(\lambda\)-homogeneous setting a new geometric issue appears. The kernels \((-\langle x,z\rangle)^\lambda\) are defined on $\Omega$ only for parameters \(z\in U_\Omega\). Consequently, for each characteristic component \(V_{\mathfrak p_\sigma}\), the Mellin construction only detects the subset $ V_{\mathfrak p_\sigma}\cap U_\Omega$, which we call the visible part of \(V_{\mathfrak p_\sigma}\) with respect to $\Omega$. The following examples illustrate why this visible part may or may not contain sufficient information for the duality argument. The discussion culminates in a visibility assumption relating the characteristic components \(V_{\mathfrak p_\sigma}\) to the cone \(\Omega\).

\begin{example}\label{ex:1}
Consider the OPC domain $\Omega=\{x\in\mathbb R^2: x_2>|x_1|\}$ and the homogeneous irreducible polynomial $P(z)=z_1-z_2$. Then $\mu=1$,
\[
    \Omega^\vee =\{\eta\in\mathbb R^2: \eta_2<-|\eta_1|\},\quad \text{and}\quad  V_{\mathfrak{p}_1} = \{z\in\mathbb C^2: z_1=z_2\}.
\]
Let $z = (u,u)\in V_{\mathfrak{p}_1}$. Then $\Re z=(\Re u,\Re u)$ and $z\in U_\Omega$ would require $\Re u<-|\Re u|$, which is impossible. Hence $V_{\mathfrak{p}_1}\cap U_\Omega=\emptyset$, so $V_{\mathfrak{p}_1}$ is not visible with respect to $\Omega$. Nevertheless, \(H_\lambda(\Omega)\cap S(\Omega)\) contains nontrivial elements. Indeed, for $\lambda\in \mathbb{C}$, the function
\begin{align}\label{soluvisu}
    \Omega \ni x\mapsto u(x) = (-\langle x,z\rangle)^\lambda = (x_1+x_2)^\lambda  \quad \text{for}\quad z=-(1,1)
\end{align}
is well defined since $x_1+x_2>0$ on $\Omega$. Moreover, $u$ is $\lambda$-homogeneous and satisfies $(\partial_1-\partial_2) u=0$ such that $u\in H_\lambda(\Omega)\cap S(\Omega)$. The corresponding parameter $z=-(1,1)$ belongs to $V_{\mathfrak{p}_1} \cap \partial U_\Omega$. Thus, although $V_{\mathfrak{p}_1}\cap U_\Omega=\emptyset$, the characteristic variety still meets the boundary of $U_\Omega$.
\end{example}

The last example raises the question of whether one should enlarge \(\Omega^\vee\) to \(\overline{\Omega^\vee}\setminus\{0\}\) in the definition of \(U_\Omega\), so that boundary parameters become visible. The next example shows that a similar effect can also be obtained by slightly shrinking $\Omega$.

\begin{example}\label{ex:2}
Consider again $P(z) = z_1 - z_2$, but now on the OPC cone $\Omega_r=\{x\in\mathbb R^2: x_2>r |x_1|\}$ for $r>0$. Then
\begin{align*}
    \Omega_r^\vee =\{\eta\in\mathbb R^2: r \eta_2<-|\eta_1|\}.
\end{align*}
For $r>1$, we have $\Omega_r\subset \Omega_1$ and the parameter $z = -(1,1)$ satisfies $z\in V\cap U_{\Omega_r}$. Thus, the solution (\ref{soluvisu}) which was only associated with a boundary parameter for $\Omega_1$, becomes visible after shrinking the cone.

For $r<1$, the situation changes in the opposite direction: we have $\Omega_r \supset  \Omega_1 $ and $z=-(1,1)$ is no longer contained in $\overline{\Omega_r^\vee}$. In this case, one can even show that $H_\lambda(\Omega_r) \cap S(\Omega_r)=\{0\}$ for $\lambda \in \mathbb{C}\setminus \mathbb{N}_0$. Indeed, every smooth solution of $(\partial_1 - \partial_2) u = 0$ is of the form $u(x_1,x_2) = F(x_1+x_2) $ for some smooth one-variable function $F$. If $r<1$, the hyperplane $\{x\in \mathbb{R}^2: x_1+x_2=0\}$ intersects the interior of $\Omega_r$. Hence, $F$ must be smooth around $0$. But a smooth $\lambda$-homogeneous function near the origin is identically zero unless $\lambda \in \mathbb{N}_0$. Therefore, $u=0$.
\end{example}

Example~\ref{ex:2} indicates that Example~\ref{ex:1} is a borderline phenomenon. After a small perturbation of the cone, either \(V_{\mathfrak p_1}\) becomes visible in \(U_\Omega\), as happens when the cone is shrunk, or \(V_{\mathfrak p_1}\) does not contribute to the \(\lambda\)-homogeneous solution space, as happens when the cone is enlarged.

As already mentioned, rather than perturbing \(\Omega\), one might try to recover solutions by allowing \(\Re z\in \overline{\Omega^\vee}\setminus\{0\}\). The next example shows that this strategy does not work in general.

\begin{example}\label{ex:3}
Consider
\begin{align*}
\Omega=\Bigl\{x\in\mathbb R^4: x_4>\sqrt{x_1^2+x_2^2+x_3^2}\Bigr\},
\qquad
P(z)=
\begin{pmatrix}
z_2-i z_1\\
z_4-z_3
\end{pmatrix}.
\end{align*}
A direct computation shows
$$
 V_{\mathfrak p_1}=\{(u,iu,v,v): u,v\in\mathbb C\},\quad \Omega^\vee=
\Bigl\{\eta\in\mathbb R^4:\eta_4< -\sqrt{\eta_1^2+\eta_2^2+\eta_3^2}\Bigr\}.
$$

First, note that $ V_{\mathfrak p_1}\cap U_\Omega = \emptyset$. Indeed, for $z=(u,iu,v,v)\in V_{\mathfrak p_1}$, the condition $\Re z\in \Omega^\vee$ becomes
\begin{align*}
    \Re v < -\sqrt{(\Re u)^2 + (\Im u)^2 + (\Re v)^2} 
\end{align*}
which is impossible. On the other hand,
\begin{align*}
V_{\mathfrak p_1}\cap \overline{U_\Omega} = \{(0,0,v,v): \Re v\leq 0\}.
\end{align*}
Thus only the slice $\{u=0\}$ of $V_{\mathfrak p_1}$ meets the boundary of $U_\Omega$. Now introduce the following coordinates adapted to $P$:
\begin{align*}
        w:=x_1+i x_2,\qquad y:=x_3+x_4.
\end{align*}

Since $x\in\Omega$ implies $y>0$, the function 

\begin{align*}
    \Omega \ni x\mapsto u(x) = (-\langle x,z\rangle)^\lambda = y^\lambda  \quad \text{for}\quad z=-(0,0,1,1)
\end{align*}
is well defined on $\Omega$. For $z=(0,0,v,v)\in V_{\mathfrak p_1}\cap \overline{U_\Omega}$, one has $-\langle x,z\rangle = (-v)\,y$, so kernels obtained from the boundary slice can only be scalar multiples of $y^\lambda$. However, the space $S(\Omega)\cap H_\lambda(\Omega)$ is much larger. Indeed, for every $k\in\mathbb{N}$,
\begin{align*}
    u_k(x):=w^k y^{\lambda-k}\in S(\Omega)\cap H_\lambda(\Omega).
\end{align*}
These functions are linearly independent and are not scalar multiples of $y^\lambda$. Thus, even after allowing boundary parameters one detects only a proper subfamily of $\lambda$-homogeneous solutions. Note that the missing modes appear as tangential derivatives along $V_{\mathfrak p_1}$ at the slice $\{u=0\}$:
\begin{align*}
        \partial_u^k\bigl(-(u w+v y)\bigr)^\lambda\Big|_{u=0} = (-1)^k(\lambda)_k\, w^k(-v y)^{\lambda-k}.
\end{align*}
\end{example}

The last example shows that adding boundary parameters alone is not sufficient. Although the slice \(V_{\mathfrak p_1}\cap \overline{U_\Omega}\) is nonempty, the corresponding boundary kernels recover only part of the \(\lambda\)-homogeneous solution space. Since the missing modes arise from tangential derivatives along \(V_{\mathfrak p_1}\), a satisfactory enlargement of the parameter set would have to include tangential jet data of arbitrarily high order as well. After slightly shrinking \(\Omega\), however, \(U_\Omega\) enlarges, and the intersection \(V_{\mathfrak p_1}\cap U_\Omega\) is relatively open in \(V_{\mathfrak p_1}\). In this case, the required jet data is already encoded in the holomorphic dependence on the parameter, since tangential derivatives can be taken directly.

We therefore take the following approach. Instead of enlarging the parameter set by adding boundary points or jets, we impose a visibility condition ensuring analytic information on each characteristic component.

\begin{assumption}[Visibility Assumption on $P$ and $\Omega$]\label{ass}
Let $\Omega$ be OPC. For every \(1\le \sigma\le \mu\) there exists a compact convex set $K^*_{\sigma}\Subset \Omega$ such that for every compact convex $\tilde{K}$ with $K^*_{\sigma}\subset \tilde{K}\Subset \Omega$, every irreducible component of the analytic set $V_{\mathfrak{p}_{\sigma}}\cap U_{\tilde{K}}$ intersects $V_{\mathfrak{p}_{\sigma}}\cap U_{\Omega}$.
\end{assumption}

\begin{remark}
(1) Note that Assumption~\ref{ass} is true for any $r\neq 1$ in Example~\ref{ex:2}.\\
(2) Since the number of characteristic components is finite, we may replace the sets \(K^*_\sigma\) by a single compact convex set. Namely, set
\begin{align}\label{Kstar}
     K^* = \operatorname{conv}(K_{1}^*\cup \dots \cup K_{\mu}^*)\Subset \Omega.
\end{align}
Then the visibility assumption holds with \(K^*\) in place of each \(K^*_\sigma\).
\end{remark}

Assumption~\ref{ass} ensures that analytic information obtained on \(V_{\mathfrak p_\sigma}\cap U_{K'}\), for \(K'\) sufficiently large, is connected to the visible part \(V_{\mathfrak p_\sigma}\cap U_\Omega\). This is precisely the property needed in the duality argument below. Further discussion of Assumption~\ref{ass} is deferred to Subsection~\ref{dis:ass}.

\subsection{Density of the homogeneous solution family}
We now turn to our main result.

\begin{theorem}\label{main}
    Let $\lambda \in \mathbb{C}\setminus \mathbb{N}_0$, let $\Omega\subset \mathbb{R}^n$ be OPC, and assume Assumption \ref{ass}. Then $\operatorname{span}( S_{\mathcal{B},\lambda}(\Omega))$ is dense in $S(\Omega)\cap H_\lambda(\Omega)^L$.
\end{theorem}

The proof proceeds by duality and is done in four steps. First, annihilation of \(S_{\mathcal B,\lambda}(\Omega)\) for $\nu\in \mathcal{E}'(\Omega)^L$ is translated, via the Mellin identity, into vanishing of $\tilde{\mathcal{B}}_{\sigma i}(z,\partial_z)M_\lambda^z\hat\nu(z)$ on \(V_{\mathfrak{p}_\sigma}\cap U_\Omega\). Second, Assumption~\ref{ass} implies that this vanishing extends to \(V_{\mathfrak{p}_\sigma}\cap U_{\tilde K}\) for some compact convex set \(\tilde K\Subset\Omega\). Third, Lemma~\ref{chito2} and Theorem~\ref{useituseit} yield a decomposition of \((E_\lambda^z)^{-1}\hat\nu\) modulo \(P^t\mathcal O(\mathbb C^n)^J\) with Paley--Wiener control. Finally, we apply \(E_\lambda^z\) and transform back to the distribution side to show that \(\nu\) annihilates \(S(\Omega)\cap H_\lambda(\Omega)^L\).

\begin{proof}
By Hahn--Banach, it suffices to prove that every $\nu \in  \mathcal{E}'(\Omega)^L$ which annihilates $S_{\mathcal{B},\lambda}(\Omega)$ also annihilates $S(\Omega)\cap H_\lambda(\Omega)^L$. Let such a $\nu$ be given, and choose a compact convex set $K\Subset \Omega$ with $\operatorname{supp}(\nu)\subset K$. By the definition of $S_{\mathcal B,\lambda}(\Omega)$, for every $1\leq \sigma\leq \mu$ and $1\leq i\leq \rho_\sigma$, we have
\begin{align}\label{saxxo}
         \tilde{\mathcal{B}}_{\sigma i}(z,\partial_z) M_\lambda^z \hat{\nu}(z) = 0\qquad \forall z\in V_{\mathfrak{p}_\sigma}\cap U_\Omega,
\end{align}
where we used (\ref{mellin-distribution-identity}) and $\Gamma(-\lambda)\neq 0$. Let \(K^*\Subset\Omega\) be the compact convex set from \eqref{Kstar}, and set $\tilde K\coloneqq \operatorname{conv}(K^* \cup K)\Subset \Omega$. We claim that for every $1\leq \sigma\leq \mu$ and $1\leq i\leq \rho_\sigma$:
\begin{align}\label{saxxo2}
         \tilde{\mathcal{B}}_{\sigma i}(z,\partial_z) M_\lambda^z \hat{\nu}(z) = 0\qquad \forall z\in V_{\mathfrak{p}_\sigma}\cap U_{\tilde K}.
\end{align}
Indeed, since \(K\subset\tilde K\), we have \(U_{\tilde K}\subset U_K\), and hence \(M_\lambda^z\hat\nu\) is holomorphic on \(U_{\tilde K}\). Fix $(\sigma,i)$, and put
\begin{align*}
      F_{\sigma i}(z):=\tilde{\mathcal{B}}_{\sigma i}(z,\partial_z)M_\lambda^z\hat\nu(z).
\end{align*}
Let \(Y\) be an irreducible analytic component of \(V_{\mathfrak p_\sigma}\cap U_{\tilde K}\). Since $F_{\sigma i}\in \mathcal{O}(U_{\tilde K})$, its restriction to $Y$ is holomorphic. By Assumption~\ref{ass}, we have \(Y\cap U_\Omega\neq\emptyset\). Since \(U_\Omega\) is open and \(U_\Omega\subset U_{\tilde K}\), the set \(Y\cap U_\Omega\) is a nonempty relatively open subset of \(Y\). On this subset, \(F_{\sigma i}\) vanishes by \((\ref{saxxo})\). Hence, by the identity theorem on irreducible analytic sets, see e.g. \cite[Ch.~1, Sec.~5.3, Corollary~2]{chirkaComplexAnalyticSets1989}, $F_{\sigma i}$ vanishes identically on \(Y\). Since \(Y\) was arbitrary, (\ref{saxxo2}) follows.

By Lemma \ref{chito2} and (\ref{saxxo2}), and since there are only finitely many pairs $(\sigma,i)$:
\begin{align*}
    \max_{1\leq \sigma \leq \mu}\max_{1\leq i\leq \rho_\sigma} \sup_{z\in V_{\mathfrak{p}_\sigma}} | \tilde{\mathcal{B}}_{\sigma i}(z,\partial_z) (E^z_\lambda)^{-1} \hat{\nu}(z)| w_{\tilde K,N}(z)<\infty\quad \text{for some } N\in \mathbb{N}_0. 
\end{align*}

By Theorem \ref{useituseit}, there exist $g\in \operatorname{PW}(\tilde K)^L \subset \operatorname{PW}(\Omega)^L$ and $\tilde{h}\in \mathcal{O}(\mathbb{C}^n)^J$ such that $(E^z_\lambda)^{-1} \hat{\nu} = g + P^t \tilde{h}$. Applying $E_\lambda^z$ gives $\hat{\nu} = E_\lambda^z g + P^t h $ for $h = E^z_{\lambda-m} \tilde{h}$ by (\ref{zweo}). Since $g, \hat{\nu}\in \operatorname{PW}(\Omega)^L$, it follows that
 \begin{align*}
       P^t h =  \hat{\nu} -  E_\lambda^z g  \in \operatorname{PW}(\Omega)^L \cap P^t \mathcal{O}(\mathbb{C}^n)^J.
 \end{align*}
Lemma~\ref{natos} therefore gives $P^t h\in P^t\operatorname{PW}(\Omega)^J$, and consequently,
\begin{align*}
     \hat{\nu} \in E_\lambda^z \operatorname{PW}(\Omega)^L + P^t \operatorname{PW}(\Omega)^J.
\end{align*}
Taking the inverse Fourier--Laplace transform, we obtain $\nu = (E_\lambda^x)^t\eta + P(\partial_x)^t \xi$ for some $\eta\in \mathcal{E}'(\Omega)^L$ and $\xi\in \mathcal{E}'(\Omega)^J$. Since $(E^x_\lambda)^t \mathcal{E}'(\Omega)^L $ annihilates $H_\lambda(\Omega)^L$ and $P(\partial_x)^t\mathcal{E}'(\Omega)^J $ annihilates $S(\Omega)$, the distribution $\nu $ annihilates $S(\Omega)\cap H_\lambda(\Omega)^L$.
\end{proof}

\section{Discussion}

\subsection{Comments on Assumption~\ref{ass}}\label{dis:ass}
Examples~\ref{ex:1}, \ref{ex:2}, and \ref{ex:3} show that
Assumption~\ref{ass} is not automatic. In each of these examples, however, the
obstruction disappears after a small perturbation of the cone. This suggests
that the visibility assumption may be mild, at least up to perturbing
\(\Omega\). It remains open whether such a perturbation argument is always possible, that is, whether there exists a sufficiently small perturbation \(\tilde \Omega\) for $\Omega$ for which Assumption~\ref{ass} holds. The author is not aware of a counterexample. It would therefore be useful either to construct one, prove that this is always true, or to find sufficient conditions for Assumption~\ref{ass}.

\subsection{Non-pointed cones \(\Omega\subset \mathbb{R}^n\)}
Recall that \(z\in U_\Omega\) guarantees \(\Re(-\langle x,z\rangle)>0\) for all \(x\in\Omega\), and hence absolute convergence of the Mellin-type integrals defining \(M_\lambda^z\). However, for \((-\langle x,z\rangle)^\lambda\) to be single-valued on \(\Omega\), one only needs \(-\langle x,z\rangle\) to avoid the branch cut \((-\infty,0]\). Thus, single-valuedness may hold even when \(z\notin U_\Omega\), as the following example shows.

\begin{example}\label{simple}
Let $0<\alpha<\pi$, and consider the cone
\begin{align*}
    \Omega_\alpha \coloneqq \{(r,\varphi): r>0,\ \varphi\in(-\alpha,\alpha)\}\subset \mathbb{R}^2\quad \text{for }(r,\varphi) \text{ polar coordinates}.
\end{align*}
For $P(\partial_x) = \Delta$, note that 
\begin{align*}
    \Omega \ni x\mapsto u(x) = (-\langle x,z\rangle)^\lambda = (x_1+ix_2)^\lambda  \quad \text{for}\quad z=-(1,i)
\end{align*}
is smooth, harmonic, and $\lambda$-homogeneous on $\Omega_\alpha$, for all $\alpha<\pi$, including the non-pointed regime $\alpha \geq \pi/2$. On the other hand, $U_{\Omega_\alpha}=\emptyset$ for $\alpha\geq \pi/2$, so in particular $z\notin U_{\Omega_\alpha}$. Compare this to the solution family for elliptic systems in \cite[Sec.~4.2]{Tsopanopoulos2025}.
\end{example}

Example \ref{simple} suggests introducing the larger parameter set
\[
\mathcal{D}_\Omega
\coloneqq
\{z\in \mathbb{C}^n : -\langle x,z\rangle \notin \mathbb{R}_{\leq 0}\ \text{for all }x\in \Omega\}.
\]
By definition, $U_\Omega \subset\mathcal{D}_\Omega$. Moreover, $-(1,i)\in \mathcal{D}_{\Omega_\alpha}$ for $\alpha<\pi$ in Example \ref{simple} since $-(1,0)\notin \Omega_\alpha$. For pointed cones, the set \(U_\Omega\) is sufficient for Theorem \ref{main}. For non-pointed cones, however, Example \ref{simple} indicates that \(\mathcal{D}_\Omega\) is a more natural parameter domain for describing solutions, leading to the enlarged family
\begin{align*}
    S^{\mathcal D}_{\mathcal{B},\lambda}(\Omega)
    \coloneqq
    \left\{
    \mathcal{B}_{\sigma i}(z,\partial_z)\bigl(-\langle x,z\rangle\bigr)^\lambda :
    \sigma=1,\dots,\mu,\;
    i=1,\dots,\rho_\sigma,\;
    z\in V_{\mathfrak{p}_\sigma}\cap \mathcal D_\Omega
    \right\}.
\end{align*}

\subsection{Graded structure and graded choices of $\mathcal{B}_{\sigma i}$}\label{gradeyou}
The ring of polynomials has a natural graded structure $\mathbb{C}[z]=\bigoplus_{d\ge 0} \mathcal{P}_d[z]$ where $\mathcal{P}_d[z]$ denotes homogeneous polynomials of degree $d$. Since $P(z)$ is homogeneous of degree $m$, the submodule $M \coloneqq P^t \mathbb C[z]^J \subset \mathbb C[z]^L$ can be decomposed into
\[
M=\bigoplus_{r\ge m} M_r,
\qquad
M_r \coloneqq M\cap \mathcal{P}_r[z]^L
     = P^t\bigl(\mathcal{P}_{r-m}[z]^J\bigr).
\]
In particular, $\mathcal{P}_d[z]\cdot M_r \subset M_{d+r}$ for all $d,r\ge 0$. Thus the module \(M\) is not merely a \(\mathbb C[z]\)-submodule, but a graded \(\mathbb C[z]\)-submodule. Since $M$ is graded, one may choose the primary decomposition $M = \bigcap_{\sigma=1}^\mu M_\sigma$ to be homogeneous; see \cite[Sec.~3.5]{E1995}. It is natural to ask whether the Noetherian operators \(\mathcal{B}_{\sigma i}(z, \partial_z) \) can be chosen compatibly with this grading. Note that 
\begin{align*}
    [E_0^z, a_\alpha \,\partial^\alpha_z] = (d-|\alpha|)a_\alpha\,\partial^\alpha_z\quad \text{for }a_\alpha\in \mathcal{P}_d[z]^L.
\end{align*}
This suggests that one should try to choose the operators \(\mathcal B_{\sigma i}\) so that all terms occurring in them have the same Euler weight. Equivalently, one would ask for
\begin{align}\label{suchachoice}
[E^z_0,\mathcal B_{\sigma i}]=\kappa_{\sigma i}\,\mathcal B_{\sigma i}\quad \text{for some }\kappa_{\sigma i}\in \mathbb{Z}.
\end{align}

The advantage of choosing the Noetherian operators in this graded form is that if $f(z)$ is homogeneous of degree $\lambda$ in $z\in \mathbb{C}^n$, then $\mathcal{B}_{\sigma i}f(z)$ is homogeneous of degree $\lambda+\kappa_{\sigma i}$. In particular, for fixed $x\in \Omega$, the function
\[
z\mapsto \mathcal B_{\sigma i}(z,\partial_z)(-\langle x,z\rangle)^\lambda
\]
is homogeneous of degree $\lambda+\kappa_{\sigma i}$ in \(z\in U_\Omega\). Hence, the existence of Noetherian operators satisfying \eqref{suchachoice} would allow one to work with equivalence classes of points in \(V_{\mathfrak p_\sigma}\cap U_\Omega\), where two points are identified when they differ by a complex scaling. Here, recall that $U_\Omega$ is invariant under real scaling, but is not a complex cone.

\subsection{Extension to $\lambda \in \mathbb{N}_0$}\label{choicelambda}
In this work, we have restricted ourselves to $\lambda \notin \mathbb{N}_0$. For $\lambda\in \mathbb{N}_0$, the kernels $\mathcal{B}_{\sigma i}(z,\partial_z) \left(-\langle x,z\rangle\right)^\lambda$ still make sense, but for fixed $z$ they produce only polynomial functions of $x$. In general, these do not suffice to describe the space $S(\Omega) \cap H_\lambda(\Omega)^L$. This is already visible in the case $\lambda=0$. Indeed, then $ \left(-\langle x,z\rangle\right)^\lambda=1$, so the resulting kernels are constant in $x$. On the other hand, for $\Omega=\{x_1>0\}\subset \mathbb{R}^2$ and $P(\partial_x) = \Delta$, the function $u(x) = \operatorname{arg}(x_1+ix_2)$ is smooth, harmonic, and $0$-homogeneous on $\Omega$, but not constant.

At present, it is not clear to the author what the most fruitful extension of the theory to $\lambda\in \mathbb{N}_0$ should be. As a starting point, one could allow additional building blocks of the form 
\begin{align*}
    (-\langle x,z\rangle)^\lambda \operatorname{arg}(-\langle x,z\rangle)\quad \text{or}\quad (-\langle x,z\rangle)^\lambda \log(-\langle x,z\rangle).
\end{align*}
The latter would lead to generalized $\lambda$-homogeneous elements, i.e. solutions lying in $\operatorname{ker}\left((E_\lambda^x)^k\right)$ for some power $k\in \mathbb{N}$, and thus to a Jordan structure for the Euler operator; see also \cite[Sec.~1.1.3]{VMR2001SPCS}.

\subsection{$\lambda$-homogeneous strong fundamental principle}
Theorem \ref{main} should be viewed as a weak $\lambda$-homogeneous fundamental principle, since it establishes density of $S_{\mathcal{B},\lambda}(\Omega)$ in \(S(\Omega) \cap H_\lambda(\Omega)^L \), but no representation formula. It is natural to conjecture that a corresponding strong \(\lambda\)-homogeneous fundamental principle also holds, namely an adaptation of the Ehrenpreis--Palamodov representation theorem of the following type: every solution \(u\in S(\Omega) \cap H_\lambda(\Omega)^L\) admits a representation
\begin{align*}
    u(x) = \sum^\mu_{\sigma=1} \sum^{\rho_\sigma}_{i=1} \int_{Z_\sigma}  \mathcal{B}_{\sigma i}(z,\partial_z) \left(-\langle x,z\rangle\right)^\lambda d\omega^{(\lambda)}_{\sigma i}(z) ,
\end{align*}
where each $\omega^{(\lambda)}_{\sigma i}$ is a suitable weighted complex Radon measure on a parameter space $Z_\sigma \subset \mathbb{C}^n$ such that the integrals converge absolutely. At present, this is as a template for future work. Several points need to be clarified, including the appropriate class of cones \(\Omega\), the correct choice of the parameter spaces \(Z_\sigma\), the measure and growth conditions on \(\omega_{\sigma i}^{(\lambda)}\), and the modification of the kernels in the exceptional case \(\lambda\in\mathbb N_0\). The present work should be regarded as a first step.

\subsection*{Contributions and Funding}
The author acknowledges Matthias Liero for helpful comments and suggestions on an earlier version of this manuscript. The author is funded by the Deutsche Forschungsgemeinschaft (DFG, German Research Foundation) under Germany's Excellence Strategy – The Berlin Mathematics Research Center MATH+ (EXC-2046/1, project ID: 390685689).

\subsection*{Disclosure of AI Use}
During the preparation and writing of this paper, the author used OpenAI's ChatGPT 4.4 and 4.5 as an AI tool. It was used to support literature search, improve language and exposition, and help identify possible errors in mathematical arguments. AI-generated output was critically reviewed and verified by the author. The author takes full responsibility for the final content of this work.


\end{document}